\documentclass[12pt]{article}
\usepackage{amsthm,amsmath,amssymb,enumerate}
\newtheorem{thm}{Theorem}[section]
\newtheorem{prop}[thm]{Proposition}
\newtheorem{lem}[thm]{Lemma}

\theoremstyle{definition}
\newtheorem*{rem}{Remark}
\newtheorem{ex}[thm]{Example}

\newcommand{\C}{\mathbb{C}}
\newcommand{\X}{\mathfrak{X}}
\newcommand{\Irr}{\operatorname{Irr}}

\newcommand{\diag}{\operatorname{diag}}

\newcommand{\spt}{\operatorname{spt}}
\newcommand{\Sym}{\operatorname{Sym}}
\newcommand{\Aut}{\operatorname{Aut}}
\newcommand{\CC}{coherent configuration}
\newcommand{\HCC}{homogeneous coherent configuration}

\begin{document}
\title{Weights on homogeneous coherent configurations}
\author{Akihide Hanaki\thanks{
    Supported by JSPS KAKENHI Grant Number JP22K03266.}}
\date{\small Faculty of Science, Shinshu University,
  Matsumoto, 390-8621, Japan\\
  {\tt hanaki@shinshu-u.ac.jp}}
\maketitle

\begin{abstract}
  D.~G.~Higman generalized a coherent configuration and defined  a weight.
  In this article, we will modify the definition and investigate weights on coherent configurations.
  If our weights are on a thin homogeneous coherent configuration,
  that is essentially a finite group, then there is a natural correspondence
  between the set of equivalence classes of weights and $2$-cohomology group
  of the group.
  We also give a construction of weights as a generalization of Higman's method
  using monomial representations of finite groups.
  \medskip

  {\it Keywords} : homogeneous coherent configuration; weight; association scheme;
  factor set; cohomology
\end{abstract}

\section{Introduction}
In \cite{Higman1975}, D.~G.~Higman established basic theory of coherent configurations,
and in \cite{Higman1976}, he generalized a coherent configuration and defined  a weight.
Typical examples of coherent configurations are defined using permutation representations of
finite groups, and weights are defined using monomial representations \cite{HigmanSapporo}.
However, after \cite{Higman1976}, weights have been little studied.
In this article, we will modify the definition and investigate weights on coherent configurations.
Especially, we will consider weights on homogeneous coherent configurations.
Homogeneous coherent configurations are just (not necessarily commutative)
association schemes in \cite{Zi}.
In Higman's definition, a weight is a generalization of a coherent configuration,
but we will consider a weight on a coherent configuration.

In Section \ref{sec2}, we will give definitions.
Also we will define equivalence of weights,
which is not considered in \cite{Higman1976}.
A finite group $G$ can be considered as a homogeneous coherent configuration.
In Section \ref{sec3}, we consider this case and show that
there is a natural correspondence
between the set of equivalence classes of weights and 2nd cohomology group of the group $G$.
In Section \ref{sec4}, we generalize the method by Higman using monomial representations
of finite groups.

\bigskip
Let $X$ be a finite set.
We denote by $M_X(\C)$ the matrix algebra
both rows and columns of whose matrices are  indexed by the set $X$.
For $a_x\in \C$ ($x\in X$), $\diag(a_x\mid x\in X)\in M_X(\C)$ is
the diagonal matrix with $(x,x)$-entry $a_x$.
For $c\subset X\times X$, $A_c\in M_X(\C)$ is defined by $(A_c)_{x,y}=1$ if $(x,y)\in c$ and $0$ otherwise.
For $c\subset X\times X$, $c^*=\{(y,x)\mid (x,y)\in c\}$.
Obviously, $A_{c^*}=A_c^T$, the transposed matrix of $A_c$.

\section{Coherent configurations and weights}\label{sec2}
\subsection{Coherent configurations}
Let $X$ be a finite set.
We consider the following conditions.
\begin{itemize}
  \item[(C1)]
  $C$ is a partition of $X\times X$.
  Namely, $C$ is a collection of non-empty subsets of $X\times X$,
  $X\times X=\bigcup_{c\in C} c$,
  and $c\cap c'=\emptyset$ for $c\ne c'\in C$.
  \item[(C2)]
  If $c\in C$, then $c^*\in C$.
  \item[(C3)]
  There is a subset $\Delta$ of $C$ such that $\bigcup_{d\in \Delta}d=\{(x,x)\mid x\in X\}$.
  \item[(C4)]
  $\C C=\bigoplus_{c\in C}\C A_c$ is an algebra.
\end{itemize}
The pair $\mathfrak{X}=(X, C)$ is called a \emph{configuration} if (C1) holds.
The configuration $\mathfrak{X}$ is said to be \emph{precoherent} if (C1), (C2)  and (C3) hold.
The configuration $\mathfrak{X}$ is said to be \emph{coherent} if (C1), (C2), (C3) and (C4) hold.
The algebra $\C C$ is called the \emph{adjacency algebra} of $\mathfrak{X}$.
When $\Delta$ is a singleton, a coherent configuration $\mathfrak{X}$ is
 said to be \emph{homogeneous}.
 Homogeneous coherent configurations are just (non-commutative) \emph{association schemes}
 in \cite{Zi}.

 \begin{ex}\label{ex2.1}[Thin {\HCC}s]
   Let $G$ be a finite group. For $g\in G$, define $c_g=\{(x,y)\in G\times G\mid xg=y\}$.
   Then $\mathfrak{X}(G)=(G,\{c_g\mid g\in G\})$ is a \HCC.
   In this case, we say that $\mathfrak{X}(G)$ is \emph{thin}.
 \end{ex}

 \begin{ex}\label{ex2.2}[Schurian \HCC, centralizer algebra]
   Let $G$ be a finite transitive permutation group on $X$,
   and let $T$ be a permutation representation of $G$.
   Set $V=\{A\in M_n(\C)\mid \text{$AT(g)=T(g)A$ for all $g\in G$}\}$, the \emph{centralizer algebra}.
   Then $V$ has a basis consisting of $01$-matrices ($G$-orbits on $X\times X$).
   We can define a homogeneous coherent configuration.
   In this case, we say that the {\HCC} is \emph{schurian}.

   For $x\in X$, the stabilizer $H=G_x$ defines the permutation representation.
   Thus we will write the {\HCC} by $\X(G,H)$.
 \end{ex}

 \subsection{Weights}
Let $\X=(X,C)$ be a \CC.
For $W\in M_X(\C)$, we define the \emph{support} of $W$ by
$$\spt(W)=\{(x,y)\in X\times X\mid W_{xy}\ne 0\}.$$
For $c\in C$, we set
$$A_c^W=A_c\circ W,$$
where $\circ$ is the entry-wise product (Hadamard product).
We consider the following conditions.
\begin{itemize}
  \item[(W1)]
  $\mathrm{spt}(W)=\bigcup_{d\in D}d$ for some subset $D$ of $C$,
  and, if $d\subset\mathrm{spt}(W)$, then $d^*\subset\mathrm{spt}(W)$.
  \item[(W2)]
  $W_{xx}\ne 0$ for all $x\in X$.
  \item[(W3)]
  $\C^W C=\bigoplus_{c\in C} \C A_c^W$  is an algebra.
  \item[(W4)]
  $W$ is hermitian,  $||W_{xy}||\in\{0,1\}$ for all $x,y\in X$, and $W_{xx}=1$ for all $x\in X$.
\end{itemize}
We call $W$ a \emph{weight} on $\X$ if (W1), (W2) and (W3) hold. 
We call $W$ an \emph{H-weight} on $\X$ if (W1), (W2), (W3) and (W4) hold
(``H-'' is due to Higman). 

\begin{rem}
  In \cite{Higman1976}, Higman called $W$ a \emph{weight} if $\X$ is precoherent configuration
  and (W1) and (W4) hold, and a \emph{coherent weight} if $W$ is a weight and
  (W3) holds (the condition (W2) is automatically holds by (W4)).
  Thus, in Higman's sense,
  weights are not necessarily {\CC}s, and
  $\C C$ is not necessarily an algebra.
  In our definition, weights are on {\CC}s and we require that $\C C$ is also an algebra.
\end{rem}

\begin{ex}
  The ``all one'' matrix $W$ is a weight on any \CC, and called the \emph{standard weight}.
  In this case, $\C^W C=\C C$.
  The identity matrix $W$ is a weight on any \CC, and called the \emph{trivial weight}.
\end{ex}

\begin{ex}\label{ex2.4}
  Let $G$ be a finite group, $H$ a subgroup of $G$.
  The induced representation from the trivial representation of $H$ to $G$ is the
  transitive permutation representation and defines a {\HCC} $\X(G,H)$ as in Example \ref{ex2.2}.
  Let $\varphi$ be a linear character of $H$, and let $U_\varphi$ be
  the representation of $H$ affording $\varphi$.
  Consider the monomial representation $U_\varphi^{\uparrow G}$ and set
  $V_\varphi=\{A\in M_n(\C)\mid 
  \text{$AU_\varphi^{\uparrow G}(g)=U_\varphi^{\uparrow G}(g)A$ for all $g\in G$}\}$.
  Then $V_\varphi$ determines an H-weight on $\X(G,H)$ \cite{HigmanSapporo}.
\end{ex}

\subsection{Equivalence}
We define an equivalence of weights on \CC.
Let $W$ and $W'$ be weights on a {\CC} $\X=(X,C)$.
Set $\Aut(\mathfrak{X})=\{\sigma\in \Sym(X) \mid \text{$c^\sigma=c$ for all $c\in C$}\}$,
where $c^\sigma=\{(x^\sigma,y^\sigma)\mid (x,y)\in c\}$.
Let $P_\sigma\in M_X(\C)$ be the permutation matrix corresponding to $\sigma\in \Sym(X)$.
We say that $W$ and $W'$ are \emph{equivalent} and write $W\sim W'$ if
there exist
$a_x\in \C^{\times}$ ($x\in X$),
$\gamma(c) \in \C^{\times}$ ($c\in C$)
and $\sigma\in\Aut(\mathfrak{X})$ such that
$$W'=\diag(a_x\mid x\in X)^{-1}P_\sigma^{-1} \sum_{c\in C} \gamma(c) A_c^W
P_\sigma\diag(a_x\mid x\in X).$$
We remark that $W=\sum_{c\in C} A_c^W$.
We say that two H-weights 
$W$ and $W'$ are \emph{H-equivalent} and write $W\sim_H W'$ if
$W\sim W'$ for 
$||a_x||=||\gamma(c)||=1$ ($x\in X$, $c\in C$)
and $\gamma(i^*)=\gamma(i)^{-1}$.

\section{Weights on a finite group}\label{sec3}
In this section, we consider weights and H-weights on a thin \HCC,
a finite group (Example \ref{ex2.1}).
We will show that weights are essentially factor sets, in this case.

We recall the basic theory of $2$-cohomology of a finite group with reference to \cite{NT}.
Let $G$ be a finite group.
A function $\alpha:G\times G \to \C^\times$ is called a \emph{factor set} or a \emph{$2$-cocycle}  if
\begin{equation}
\alpha(g,h)\alpha(gh,k)=\alpha(g,hk)\alpha(h,k) \tag{$\ast$}
\end{equation}
hold for all $g,h,k\in G$.
It is possible to change $\C^\times$ to an abelian group $M$ and
consider the action of $G$ on $M$,
but we will consider only the trivial action on $\C^\times$.
If we consider a vector space $\bigoplus_{g\in G}\C v_g$ with the multiplication
$v_g v_h=\alpha(g,h)v_{gh}$, then the condition ($\ast)$ is equivalent
to the associativity of the multiplication.
Thus, for a factor set $\alpha$, we can define the \emph{generalized group algebra}
$\C^{(\alpha)}G=\bigoplus_{g\in G}\C v_g$.
The set $Z^2(G,\C^\times)$ of all factor sets is an abelian group
by $(\alpha\beta)(g,h)=\alpha(g,h)\beta(g,h)$.
Changing the basis $v_g\mapsto \gamma(g)v_g$ yields the change of the factor set
$\alpha(g,h)\mapsto \gamma(g)\gamma(h)\gamma(gh)^{-1}\alpha(g,h)$.
We say that factor sets $\alpha$ and $\beta$ are \emph{equivalent} and write $\alpha\sim \beta$
if there exists a map $\gamma:G\to\C^\times$ such that
$\beta(g,h)= \gamma(g)\gamma(h)\gamma(gh)^{-1}\alpha(g,h)$.
A factor set $\alpha$ is called a \emph{$2$-coboundary} if 
there exists a map $\gamma:G\to\C^\times$ such that
$\alpha(g,h)= \gamma(g)\gamma(h)\gamma(gh)^{-1}$.
The set $B^2(G,\C^\times)$ of all $2$-coboundaries is a subgroup of $Z^2(G,\C^\times)$.
The factor group $H^2(G,\C^\times)=Z^2(G,\C^\times)/B^2(G,\C^\times)$ is called the
\emph{$2$-cohomology group} of $G$.
An element of $H^2(G,\C^\times)$ is an equivalent class of factor sets.

\begin{prop}[{\cite[II. Section 7.2, Theorem 7.3, III. Lemma 5.4]{NT}}]\label{fs_normalized}
  For a factor set $\alpha$, there exists a factor set $\beta$ equivalent to $\alpha$
  satisfying the following conditions.
  \begin{enumerate}[(1)]
    \item $\beta(g,1)=\beta(1,g)=1$, $\beta(g,g^{-1})=\beta(g^{-1},g)=1$ for all $g\in G$.
    \item 
    $||\beta(g,h)||=1$ for all $g,h\in G$.
  \end{enumerate}
\end{prop}

Now we consider weights on thin {\HCC}s.
The next lemma is clear by definition.

\begin{lem}
  Let $W$ be a weight on a thin {\HCC} $\X(G)$ defined by a finite group $G$.
  Then $\spt(W)=\bigcup_{h\in H}c_h$ for some subgroup $H$ of $G$,
  where $c_h=\{(x,y)\in G\times G\mid xh=y\}$.
\end{lem}

We say that $W$ has ``no zero entry'', if $\spt(W)=G$.
Clearly, if $W\sim W'$ and $W$ has no zero entry, then so does $W'$.
The following theorems are the main results in this section.

\begin{thm}\label{thmA}
  Let $\X(G)$ be a thin {\HCC}  defined by a finite group $G$.
  Then there exist natural bijections between the following sets:
  \begin{enumerate}[(1)]
    \item $\mathfrak{W}$ : the set of $\sim$ equivalence classes of weights
    on $\X(G)$ having no zero entry,
    \item $\mathfrak{W}_H$ : the set of $\sim_H$ equivalence classes of H-weights
    on $\X(G)$ having no zero entry,
    \item $H^2(G,\C^\times)$ : the $2$-cohomology group of $G$.
  \end{enumerate}
\end{thm}

\begin{thm}\label{thmB}
  Let $\X(G)$ be a thin {\HCC} defined by a finite group $G$.
  Then there exist natural bijections between the following sets:
  \begin{enumerate}[(1)]
    \item the set of $\sim$ equivalence classes of weights on $\X(G)$,
    \item the set of $\sim_H$ equivalence classes of H-weights on $\X(G)$,
    \item $\bigcup_{H} H^2(H,\C^\times)$, where $H$ runs over all subgroups of $G$.
  \end{enumerate}
\end{thm}

To prove the above theorems, we need some lemmas.
For a moment, we suppose that $W$ is a weight of $\X(G)$ having no zero entry.
Set $c_g=\{(x,y)\in G\times G\mid xg=y\}$ and $A_g=A_{c_g}$, the adjacency matrix.
By
$$A_g^W A_h^W=\alpha_W(g,h)A_{gh}^W$$
for $g,h\in G$, we obtain a factor set $\alpha_W$. 
For a factor set $\alpha$, we can define $W_\alpha\in M_G(\C)$ by
$$(W_\alpha)_{xy}=\alpha(x,x^{-1}y).$$

\begin{lem}\label{lem0}
  For a factor set $\alpha$, $W_\alpha$ is a weight and $\alpha_{W_\alpha}=\alpha$.
\end{lem}

\begin{proof}
  By definition, 
  \begin{eqnarray*}
    (A_g^{W_\alpha} A_h^{W_\alpha})_{xy}
    &=& \sum_{z\in G}\delta_{xg,z}\alpha(x,x^{-1}z)\delta_{zh,y}\alpha(z,z^{-1}y)
        =\delta_{xgh,y}\alpha(x,g)\alpha(xg,h)\\
    &=& \delta_{xgh,y}\alpha(x,gh)\alpha(g,h)\\
    (A_{gh}^{W_\alpha})_{xy} &=& \delta_{xgh,y}\alpha(x,x^{-1}y)=\delta_{xgh,y}\alpha(x,gh).
  \end{eqnarray*}
  Thus $A_g^{W_\alpha} A_h^{W_\alpha}=\alpha(g,h)A_{gh}^{W_\alpha}$ and the result holds.
\end{proof}

\begin{lem}\label{lem1}
  Let $W$ be a weight on a thin {\HCC} $\X(G)$ defined by a finite group $G$.
  Then
  $$\alpha_W(g,h)=\frac{W_{x,xg}W_{xg,xgh}}{W_{x,xgh}}$$
  for all $g,h,x\in G$.
  In particular, if $W$ is an H-weight, then $||\alpha_W(g,h)||=1$ for all $g,h\in G$.
\end{lem}

\begin{proof}
  The result follows from
    \begin{eqnarray*}
    (A_g^W A_h^W)_{xy} &=& \sum_{z\in G}\delta_{xg,z}W_{x,z}\delta_{zh,y}W_{z,y}\\
                        &=& \delta_{xgh,y}W_{x,xg}W_{xg,xgh},\\
    (A_g^W A_h^W)_{xy} &=& \alpha_W(g,h)(A_{gh}^W)_{x,y}
                            = \alpha_W(g,h)\delta_{xgh,y}W_{x,xgh}
  \end{eqnarray*}
  for $g,h,x,y\in G$.
\end{proof}

For a factor set $\alpha$, a weight $W$, and an H-weight $W$,
we denote the $\sim$ equivalence class containing $\alpha$ by $[\alpha]$,
the $\sim$ equivalence class containing $W$ by $[W]$,
and the $\sim_H$ equivalence class containing $W$ by $[W]_H$, respectively.
We define
$$\Phi:H^2(G,\C^\times)\to \mathfrak{W}, \quad \Phi([\alpha])=[W_\alpha]$$
and show that $\Phi$ is a bijection.

\begin{lem}\label{lem2}
  For a weight $W$, $W\sim W_{\alpha_W}$.
\end{lem}

\begin{proof}
  Set $a_g=W_{1g}^{-1}$ ($g\in G$), and set
  $$W'=\diag(a_g\mid g\in G)^{-1} W \diag(a_g\mid g\in G).$$
  Then $\alpha_W=\alpha_{W'}$, $W\sim W'$, and 
  $W'_{1g}=W'_{11}$ ($g\in G$).
  By setting $x=1$ in Lemma \ref{lem1}, we have $\alpha_W(g,h)=W'_{g,gh}$
  and thus $W'=W_{\alpha_W}$.
%
\end{proof}

\begin{lem}\label{lemPhi1}
  For factor sets $\alpha$ and $\beta$,
  if $\alpha\sim\beta$, then $W_\alpha\sim W_\beta$.
  (Namely, $\Phi$ is well-defined.)
\end{lem}

\begin{proof}
  Suppose $\beta(g,h)=\gamma(g)\gamma(h)\gamma(gh)^{-1}\alpha(g,h)$.
  We have  $W_\alpha=\sum_{g\in G} A_g^{W_\alpha}$ and set
  $W'=\sum_{g\in G} \gamma(g)A_g^{W_\alpha}$.
  Then $W'\sim W_\alpha$ and $W'$ is a weight with the factor set $\beta$.
  By Lemma \ref{lem2}, $W'\sim W_\beta$ and
  $W_\alpha\sim W_\beta$.
\end{proof}

\begin{lem}\label{lemPhi2}
  The map $\Phi$ is surjective.
\end{lem}

\begin{proof}
  This is clear by Lemma \ref{lem2}.
\end{proof}

\begin{lem}\label{lemPhi3}
  For factor sets $\alpha$ and $\beta$,
  if $W_\alpha\sim W_\beta$, then  $\alpha\sim \beta$.
  (This means that $\Phi$ is injective.)
\end{lem}

\begin{proof}
  For $W_\alpha=\sum_{g\in G}A_g^{W_\alpha}$, we can write
  $$W_\beta=\diag(a_g\mid g\in G)^{-1}P_\sigma^{-1} \sum_{g\in G}\gamma(g)
  A_g^{W_\alpha}P_\sigma\diag(a_g\mid g\in G).$$
  Thus the factor set obtained by $W_\beta$ is $\gamma(g)\gamma(h)\gamma(gh)^{-1}\alpha(g,h)$.
  We have $\alpha\sim \beta$.
\end{proof}

By Lemmas \ref{lemPhi1}, \ref{lemPhi2}, \ref{lemPhi3}, $\Phi:H^2(G,\C^\times)\to \mathfrak{W}$
is bijective.

We consider
$$\Psi:\mathfrak{W}_H\to \mathfrak{W},\quad \Psi([W]_H)=[W]$$
and show that $\Psi$ is a bijection.
By definition, it is clear that $\Psi$ is well-defined,
namely, if $W\sim_H W'$, then $W\sim W'$.

\begin{lem}\label{lemPsi1}
  If a factor set $\alpha$ satisfies the conditions (1) and (2) in Proposition \ref{fs_normalized},
  then  $W_\alpha$ is an H-weight.
\end{lem}

\begin{proof}
  Suppose that $\alpha$ satisfies (1) and (2).
  Recall that $(W_\alpha)_{xy}=\alpha(x,x^{-1}y)$.
  Since $\alpha(x,1)=1$, $(W_\alpha)_{xx}=1$ for $x\in G$.
  For all $x,y\in G$, 
  $||(W_\alpha)_{xy}||=||\alpha(x,x^{-1}y)||=1$ hold.
  By $\alpha(x,x^{-1})=1$, $A_{x^{-1}}^{W_\alpha}=(A_{x}^{W_\alpha})^{-1}=(A_{x}^{W_\alpha})^*$.
  Thus $W=\sum_{x\in G} A_{x}^{W_\alpha}$ is hermitian.
  Now $W$ is an H-weight.
\end{proof}

\begin{lem}\label{lemPsi2}
  $\Psi$ is surjective.
\end{lem}

\begin{proof}
  Let $W$ be a weight.
  By Lemma \ref{lem2}, $W\sim W_{\alpha_W}$.
  By Proposition \ref{fs_normalized}, there exists a factor set $\beta\sim\alpha_W$ which satisfies
  the conditions (1) and (2) in Proposition \ref{fs_normalized}.
  By Lemmas \ref{lemPhi1},  \ref{lemPsi1},
  $W_\beta$ is an H-weight and $W_{\alpha_W}\sim W_\beta$.
  Now $\Psi([W_\beta]_H)=[W_\beta]=[W]$ and $\Psi$ is surjective.
\end{proof}

\begin{lem}\label{lemPsi3}
  For H-weights $W$ and $W'$, if $W\sim W'$, then $W\sim_H W'$.
  (This means that $\Psi$ is injective.)
\end{lem}

\begin{proof}
  Suppose that $W$ and $W'$ are H-weights and set
  $$W'=\diag(a_g\mid g\in G)^{-1}P_\sigma^{-1} \sum_{g\in G}\gamma(g)
  A_g^{W}P_\sigma\diag(a_g\mid g\in G).$$
  We have
  $\alpha_{W'}(g,h)=\gamma(g)\gamma(h)\gamma(gh)^{-1}\alpha_W(g,h)$ and
  $||\alpha_W(g,h)||=||\alpha_{W'}(g,h)||=1$ for all $g,h\in G$.
  Choose $g_0\in G$ such that $||\gamma(g_0)||$ is maximal.
  By $\alpha_{W'}(g_0,g_0)=\gamma(g_0)^2\gamma(g_0^2)^{-1}\alpha(g_0,g_0)$,
  we have $||\gamma(g_0)||^2=||\gamma(g_0^2)||\leq ||\gamma(g_0)||$ and this shows
  $||\gamma(g_0)||\leq 1$.
  Conversely, choose $g_1\in G$ such that $||\gamma(g_1)||$ is minimal.
  Then we have $||\gamma(g_1)||\geq 1$.
  This shows that $||\gamma(g)||=1$ for all $g\in G$.

  The absolute values of entries of $P_\sigma^{-1} \sum_{g\in G}\gamma(g)A_g^{W}P_\sigma$ are $1$.
  For all $g,h\in G$, we can see that $||a_g^{-1}a_h||=1$ and
  this shows $||a_g||=1$ for all $g\in G$.
  Now $W\sim_H W'$.
\end{proof}

\begin{proof}[Proof of Theorem \ref{thmA}]
  By Lemmas \ref{lemPhi1}, \ref{lemPhi2}, \ref{lemPhi3}, $\Phi:H^2(G,\C^\times)\to \mathfrak{W}$
  is bijective.
  By Lemmas \ref{lemPsi2},  \ref{lemPsi3}, $\Psi:  \mathfrak{W}_H \to \mathfrak{W}$
  is bojective.
\end{proof}

We consider weights having zero entry.

\begin{proof}[Proof of Theorem \ref{thmB}]
  Let $W$ be a weight on $\X(G)$ which can have zero entry.
  The support of $W$ determines a subgroup $H$ of $G$.
  We can write $W$ as a block diagonal matrix, and all blocks are weights of $H$
  of the same factor set of $H$.
  We remark that the support is invariant under equivalence of weights.
  Now the result holds.
\end{proof}

\section{A construction}\label{sec4}
In this section, we will generalize the result by Higman \cite{HigmanSapporo}
to construct weights on \HCC.
Higman used monomial representations of finite groups.
We will define monomial representations of \HCC,
and construct weights on \HCC.
As we mentioned, {\HCC}s are association schemes in the sense in \cite{Zi}.
We will use terminologies in \cite{Zi}.

We summarize the theory of {\HCC}s (association schemes) and their representations
with reference to \cite{Higman1975,Zi}.

Let $\X=(X,C)$ be a \HCC.
The adjacency algebra $\C C$ is known to be semisimple.
Thus we can write $\C C\cong\bigoplus_{i=1}^r M_{n_i}(\C)$,
and character theory works well.
By $\Irr(C)$, we denote the set of all irreducible characters of $\C C$.
Naturally $\C C$ acts on $\C X$, and we call $\C X$ the \emph{standard module}.
The \emph{standard character} is also defined.
The multiplicity of $\chi\in \Irr(C)$ in the standard character is called the \emph{multiplicity}
of $\chi$ and denoted by $m_\chi$.
There is a natural $\C C$-monomorphism from the regular module $\C C$
to the standard module $\C X$,
and thus $\chi(1)\leq m_\chi$ holds for $\chi\in \Irr(C)$.
Let $e_\chi$ be the central primitive idempotent corresponding to $\chi\in \Irr(C)$.
Then the rank of $e_\chi$ as an element in $M_X(\C)$ is $m_\chi \chi(1)$.

A subset $D$ of $C$ is called a \emph{closed subset} of $C$ (or $\X$) if
$\C D=\bigoplus_{d\in D}\C A_d$ is a subalgebra of $\C C$.
Suppose that $D$ is a closed subset.
For $x\in X$, set $xD=\{y\in X\mid \text{$(x,y)\in d$ for some $d\in D$}\}$.
We have a partition $X=x_1D\cup\dots\cup x_mD$.
Then $(x_iD, D_{x_iD})$ is also a {\HCC}, called a \emph{sub \HCC},
where $D_{x_iD}=\{d\cap(x_iD\times x_iD)\mid d\in D\}$.
We remark that sub {\HCC}s are not necessarily isomorphic,
but their adjacency algebras are isomorphic to $\C D$. 
Thus we can identify $\Irr(D_{x_iD})$ and write $\Irr(D)$.
Set $X/D=\{x_1D,\dots,x_mD\}$.
Define $c^D=\{(x_iD,x_jD)\mid c\cap(x_iD\times x_jD)\ne \emptyset\}$
for $c\in C$ and $C/\!/D=\{c^D\mid c\in C\}$.
Then $(X/D, C/\!/D)$ is a {\HCC}, called the \emph{factor \HCC}.
We remark that $C/\!/D$ defines a partition of $C$.

Now we define monomial representations (character) of {\HCC}s.
Let $D$ be a closed subset of $\X=(X,C)$,
and let $\varphi\in\Irr(D)$ be of multiplicity one.
The induced character $\varphi^{\uparrow C}$ is called a \emph{monomial character} of $\X$.
Let $e_i\in M_{x_iD}(\C)$ be the central primitive idempotent corresponding to $\varphi\in \Irr(D)$.
Since $m_\varphi=1$, the rank of $e_i$ is $1$.
By \cite[Theorem 2.8]{Hirasaka-Muzychuk2002},
$\varphi$ is essentially a character of a cyclic group,
namely, there exists a closed subset $K$ of $C$ such that the factors of sub {\HCC}s are
isomorphic cyclic groups.
Thus we may assume that all $e_i$ are same matrices $e$.
Now we can set
$$e_\varphi=
\begin{pmatrix}
  e && \\
    &\ddots&\\
  &&e
\end{pmatrix}\in \C D\subset \C C\subset M_X(\C),$$
the primitive idempotent in $\C D$ corresponding to $\varphi\in \Irr(D)$.
The $\C D$-module $e_\varphi \C D$ affords the character $\varphi$, and so the
induced module, the module of monomial representation, is
$$e_\varphi \C D_{\C D}\otimes \C C\cong e_\varphi \C C.$$

We consider the endomorphism algebra
$\mathrm{End}_{\C C}(e_\varphi \C C)\cong e_\varphi \C C e_\varphi$.
We set $|x_iD|=\ell$.
Then the rank of $e M_\ell(\C) e =1$ and so $e M_\ell(\C) e=\C e$.
Thus
$$e_\varphi \C C e_\varphi\subset e_\varphi M_X(\C)e_\varphi\cong 
\begin{pmatrix}
  \C e & \dots &\C e\\
       & \dots & \\
  \C e & \dots &\C e\\
\end{pmatrix}\cong M_m(\C)\cong M_{X/D}(\C).$$
We can define an algebra homomorphism $\Gamma:e_\varphi \C C e_\varphi\to M_{X/D}(\C)$.

We will choose representatives of $c^D$.

\begin{lem}\label{lem4.1}
  Suppose $e_\varphi A_ce_\varphi\ne 0$.
  Then, for $c'\in C$ with $c^D=c'^D$,  there exists $\mu\in \C$ such that
  $e_\varphi A_{c'}e_\varphi=\mu e_\varphi A_c e_\varphi$.
\end{lem}

\begin{proof}
  Write
  $$A_c
  =\begin{pmatrix}
    (A_c)_{11}&\dots&(A_c)_{1m}\\
              &\dots&\\
    (A_c)_{m1}&\dots&(A_c)_{mm}
  \end{pmatrix}, \quad 
  A_{c'}
  =\begin{pmatrix}
    (A_{c'})_{11}&\dots&(A_{c'})_{1m}\\
              &\dots&\\
    (A_{c'})_{m1}&\dots&(A_{c'})_{mm}
  \end{pmatrix}.$$
  Then we can write
  $$e_\varphi A_c e_\varphi
  =\begin{pmatrix}
    e(A_c)_{11}e&\dots&e(A_c)_{1m}e\\
              &\dots&\\
    e(A_c)_{m1}e&\dots&e(A_c)_{mm}e
  \end{pmatrix}
  =\begin{pmatrix}
    a_{11}e&\dots&a_{1m}e\\
              &\dots&\\
    a_{m1}e&\dots&a_{mm}e
  \end{pmatrix},$$
  $$e_\varphi A_{c'} e_\varphi
  =\begin{pmatrix}
    e(A_{c'})_{11}e&\dots&e(A_{c'})_{1m}e\\
              &\dots&\\
    e(A_{c'})_{m1}e&\dots&e(A_{c'})_{mm}e
  \end{pmatrix}
  =\begin{pmatrix}
    b_{11}e&\dots&b_{1m}e\\
              &\dots&\\
    b_{m1}e&\dots&b_{mm}e
  \end{pmatrix}$$
  for some $a_{ij}, b_{ij}\in \C$.
  Suppose $a_{st}\ne 0$. 
  Set $L=a_{st}^{-1}b_{st}e_\varphi A_{c} e_\varphi - e_\varphi A_{c'} e_\varphi$.
  Then the $(s,t)$-part of $L$ is $0$.
  We remark that
  \begin{eqnarray*}
    (x_i D,x_j d)\in c^D &\Longleftrightarrow& (A_c)_{ij}\ne 0,\\
     (x_i D,x_j d)\not\in c^D &\Longrightarrow& e(A_c)_{ij}e= 0,\ a_{ij}=0.
  \end{eqnarray*}
  We put $U=\{c_1\in C\mid c_1^D=c^D\}$.
  Since $L\in \C C$,  we can write
  $L=\sum_{c_1\in U}\mu(c_1)A_{c_1}$ for some $\mu(c_1)\in \C$.
  By the definition of $c^D$, every $A_{c_1}$ ($c_1\in U$) has non-zero entries in the $(s,t)$-part,
  and thus $\mu(c_1)=0$ for all $c_1\in U$.
  Now $e_\varphi A_{c'} e_\varphi=a_{st}^{-1}b_{st}e_\varphi A_{c} e_\varphi$.
\end{proof}

Choose $c_\lambda\in C$ ($\lambda\in\Lambda$) such that
$C/\!/D=\{c_\lambda^D\mid \lambda\in\Lambda\}$,
$c_\lambda^D\ne c_{\lambda'}^D$ if $\lambda\ne \lambda'$,
and $e_\varphi A_{c_\lambda} e_\varphi\ne 0$ if such $c_\lambda$ exists.
Then $\{e_\varphi A_{c_\lambda}e_\varphi
\mid \lambda\in\Lambda, \ e_\varphi A_{c_\lambda} e_\varphi\ne 0\}$
is a basis of $e_\varphi \C C e_\varphi$.
We put
$$W=\sum_{\lambda\in\Lambda}\Gamma(e_\varphi A_{c_\lambda} e_\varphi)$$
and show that $W$ is a weight on $(X/D, C/\!/D)$.
We remark that $c_\lambda$ is not unique to $c_\lambda^D$, but
$\Gamma(e_\varphi A_{c_\lambda} e_\varphi)$ is unique up to scalar multiple by Lemma \ref{lem4.1}
and thus $W$ is unique up to equivalence of weights.

\begin{thm}
  Let $\X=(X,C)$ be a {\HCC}, and $D$ a closed subset of $C$.
  Let $\varphi$ be an irreducible character of $D$ of multiplicity one.
  Then $W$ defined above is a weight on the factor {\HCC} $(X/D, C/\!/D)$.
\end{thm}

\begin{proof}
  It is easy to see that
  $\spt(W)=\bigcup_{c}c^D$, where $c$ runs over $\{c\in C\mid e_\varphi A_c e_\varphi\ne 0\}$.
  Suppose $c^D\subset \spt(W)$. We may assume $e_\varphi A_c e_\varphi\ne 0$.
  We remark that $e$ is hermitian, because $e$ is essentially a central primitive idempotent
  corresponding to a linear character of a finite group.
  We have $e_\varphi A_{c^*}e_\varphi=(e_\varphi A_c e_\varphi)^*\ne 0$, and so $c^*\subset \spt(W)$.
  The condition (W1) holds.
  
  (W2) is clear.
  (W3) is also clear since $\C^{W}(C/\!/D)$
  is the image of the algebra homomorphism $\Gamma$.
\end{proof}

\bibliographystyle{amsplain}
\providecommand{\bysame}{\leavevmode\hbox to3em{\hrulefill}\thinspace}
\providecommand{\MR}{\relax\ifhmode\unskip\space\fi MR }
\providecommand{\MRhref}[2]{%
  \href{http://www.ams.org/mathscinet-getitem?mr=#1}{#2}
}
\providecommand{\href}[2]{#2}

\end{document}